\documentclass{amsart}
\usepackage{amsmath,amscd, amssymb}
\usepackage{color}
\usepackage[colorlinks]{hyperref}
\usepackage{tikz}
\usetikzlibrary{arrows}
\newtheorem{thm}{Theorem}[section]
\newtheorem{cor}[thm]{Corollary}
\newtheorem{lem}[thm]{Lemma}

\theoremstyle{definition}
\newtheorem{example}[thm]{Example}
\theoremstyle{definition}
\newtheorem{defn}[thm]{Definition}
\theoremstyle{remark}
\newtheorem{rem}[thm]{Remark}
\numberwithin{equation}{section}

\begin{document}

\title{On semi-open codes and bi-continuing almost everywhere codes}
\author{D. Ahmadi Dastjerdi and S. Jangjooye Shaldehi }%
\address{Faculty of Mathematics , University of Guilan}
\email{dahmadi1387@gmail.com, sjangjoo90@gmail.com}%


\begin{abstract}
 We will show that a system is synchronized if and only if it has a cover whose cover map is semi-open.
Also, any factor code on an irreducible sofic shift is semi-open and the image of a synchronized system by a semi-open code is synchronized.  On the other side, right-closing semi-open extension of an irreducible shift of finite type is of finite type.
 Moreover, we give conditions on finite-to-one factor codes to be open and show that any semi-open code on a synchronized system is bi-continuing a.e.. We give some sufficient conditions for a right-continuing a.e. factor code being right-continuing everywhere.
\end{abstract}
\maketitle
\section{Introduction}

 A map $ \phi:X\rightarrow Y $ is called  \emph{open map} if images of open sets are open and 
 is called \emph{semi-open (quasi-interior)} if images of open sets have non-empty interior. 
In dynamical systems, they are of interest when they appear as factor map or as it is called factor code in symbolic dynamics.
There are many classes of open and specially semi-open factor maps. For instance, any factor map between minimal compacta is semi-open \cite{Au} and all surjective
cellular automata are semi-open \cite[Theorem 2.9.3]{SM}.
As one may expect, there are some strict restrictions for a factor map being open; nonetheless, some important classes do exist. As an example, if $X$ and $Y$ are compact minimal spaces, there are almost one-to-one minimal extensions $X'$ and $Y'$ by $\psi_1$ and $\psi_2$ respectively and an open factor map $ \phi':X'\rightarrow Y' $ such that $ \psi_{2}\circ\phi'=\phi\circ\psi_1 $ \cite{A2}. Here, our investigation is
mainly confined on semi-open factor codes between shift spaces. For open factor codes between shift spaces see \cite{J2, J1}. 
A stronger notion to a semi-open map is the notion of an irreducible map which will be of our interest as well. A map
$  \phi:X\rightarrow Y$ is called \emph{irreducible} if the only closed set $A\subseteq X  $ for which $ \phi(A)=\phi(X) $ is
$A=X  $. 

A summary of our results is as follows. In section \ref{properties}, we show that a coded system $X$ is synchronized if and only if there is a cover 
$\mathcal G=(G,\,\mathcal L)$ so that its cover map $\mathcal L_{\infty}$ is semi-open and if $ \mathcal G=(G, \mathcal L) $ is
an irreducible right-resolving cover for $X$ with a magic word, then $ \mathcal L_{\infty} $ is semi-open (Theorem \ref{r-r magic}). Also, when $ \mathcal G $ is Fischer cover, $ \mathcal L_{\infty} $ is irreducible (Theorem \ref{Fischer}). 
Furthermore, any factor code on an irreducible sofic shift is semi-open (Corollary \ref{new}). Theorem \ref{right closing} implies that a right-closing semi-open extension of an irreducible shift of finite type is of finite type. 
Theorem \ref{s to s} shows that the image  of a synchronized system  is  again synchronized under a semi-open code and so providing a class of non-semi-open codes, mainly, codes factoring synchronized systems over non-synchronized ones.

We also consider bi-continuing and bi-continuing almost everywhere (a.e.) codes. A right-continuing or u-eresolving code is a code which is surjective on
each unstable set and it is a natural dual version of a right-closing
code and plays a fundamental role in the class of infinite-to-one codes \cite{Boyle}. In Theorem \ref{bi-retract}, we give conditions on bi-continuing a.e. factor codes to be open and also by Corollary \ref{cor} and Theorem \ref{Ballier}, we give some sufficient conditions for a right-continuing a.e. factor code being right-continuing everywhere. Finally, we show that any semi-open code on a synchronized system is bi-continuing a.e. (Theorem \ref{syn bi-continuing a.e.}).

\section{Background and Notations}

Let $\mathcal A$ be a non-empty finite set. The full $\mathcal
A$-shift
 denoted by ${\mathcal A}^{\mathbb Z}$, is the collection of all bi-infinite sequences of symbols from ${\mathcal A}$.
A block (or word) over ${\mathcal A}$ is a finite sequence of
symbols from ${\mathcal A}$. The \emph{shift map} on ${\mathcal A}^{\mathbb Z}$ is the map $\sigma$ where
$\sigma(\{x_i\})=\{y_i\}$ is defined by $y_{i}=x_{i+1}$. The pair $(\mathcal {A}^{\mathbb Z},\,\sigma)$ is called the \emph{full shift} and any closed invariant set of that is called a \emph{shift space} over $\mathcal{A}$.

Denote by ${\mathcal B}_{n}(X)$  the set of all admissible $n$-words and let ${\mathcal B}(X)=\bigcup_{n=0}^{\infty}{\mathcal B}_{n}(X)$
be  the \emph{language} of $X$.
For $ u\in\mathcal B(X) $, let the \emph{cylinder}
$ [u] $ be the set $ \{x\in X:\,x_{[l,l+|u|-1]}=u\} $. For $l\geq0  $ and  $ |u|=2l+1 $, $ [u] $ is called a \emph{central} $2l+1  $
cylinder. 

Let ${\mathcal A}$ and ${\mathcal D}$ be alphabets and $X$ a
subshift over ${\mathcal A}$. For $m,\,n\in \mathbb Z$ with $-m
\leq n$, define the \emph{$(m+n+1)$-block map} $\Phi: {\mathcal
B}_{m+n+1}(X) \rightarrow {\mathcal D}$ by
\begin{equation}\label{2.1} 
y_{i}=\Phi(x_{i-m}x_{i-m+1}...x_{i+n})=\Phi(x_{[i-m,i+n]})
\end{equation}
where $y_{i}$ is a symbol in ${\mathcal D}$. The map $\phi=\Phi_{\infty}^{[-m,n]}: X
\rightarrow {\mathcal D}^{\mathbb Z}$ defined by $y=\phi(x)$ with
$y_{i}$ given by \ref{2.1} is called the \emph{code} induced by
$\Phi$. If $ m=n=0 $, then $ \phi $ is called \emph{$ 1 $-block code} and $ \phi=\Phi_{\infty} $. An onto code $\phi: X \rightarrow Y$ is called a
\emph{factor code}.

A point $  x$ in a shift space $  X$ is \emph{doubly transitive} if every
block in $  X$ appears in $  x$ infinitely often to the left and to the right. Let $ \phi:X\rightarrow Y $ be a factor code. If
there is a positive integer $  d$ such that every doubly transitive point of $  Y$
has exactly $  d$ pre-images under $ \phi $, then we call $  d$ the \emph{degree} of $ \phi $.

A code $ \phi:X\rightarrow Y  $ is called \emph{right-closing} (resp. \emph{right-continuing}) if whenever $x\in X  $, $ y\in Y $ and $ \phi(x) $ is left
asymptotic to $  y$, then there exists at most (resp. at least) one $ \overline{x}\in X $ such that $ \overline{x} $ is left asymptotic
to $  x$ and $ \phi(\overline{x})=y $. A \emph{left-closing} (resp. \emph{left-continuing}) code
is defined similarly. If $  \phi$ is both left and right-closing (resp. continuing), it is called \emph{bi-closing} (resp. \emph{bi-continuing}). An integer $ n\in\mathbb Z^{+} $ is called a \emph{(right-continuing) retract} of a right-continuing code $ \phi:X\rightarrow Y  $ if, whenever $x\in X  $ and $ y\in Y $ with $ \phi(x)_{(-\infty,0]}=y_{(-\infty,0]} $, we can find $ \overline{x}\in X $ such that $ \phi(\overline{x})=y $ and $ x_{(-\infty,-n]}=\overline{x}_{(-\infty,-n]} $ \cite{J2}.

 Let $G$ be a directed graph and $  \mathcal V$ (resp. $\mathcal E$) the set of its vertices (resp. edges) which is supposed to be countable.
 An \emph{edge shift}, denoted by $X_{G}$,  is a shift space which consist of all
  bi-infinite sequences of edges from $ \mathcal E $. A graph $G$ is called \emph{locally finite}, if it has finite out-degree and finite in-degree at any vertex. Recall that $X_G$ is locally compact if and only if $G$ is locally finite.

Let $v\sim w  $ be an  equivalence relation
on $  \mathcal V$ whenever there is a path
from $  v$ to $  w$ and vice versa. For an 
equivalence class, consider all vertices together with
all edges whose endpoints are in that equivalence class. Thus a subgragh called the (irreducible) \emph{component} of $G$ associated to that class arises; and if it is non-empty, then an irreducible subshift
of $X_{G}$, called (irreducible)
component of $X_{G}$ is defined.
There is no subshift of $X_{G}$ which is irreducible
and contains a component of $X_{G}$ properly.

A labeled graph ${\mathcal G}$ is a pair $(G,{\mathcal L})$ where
$G$ is a graph and  
${\mathcal L}: {\mathcal E} \rightarrow {\mathcal A}$ its labeling. Associated to $\mathcal G$, a space 
$$ X_{{\mathcal G}}=\text{closure}\{{\mathcal L}_{\infty}(\xi): \xi \in X_{G}\}=\overline{{\mathcal L}_{\infty}(X_{G})}$$
is defined and ${\mathcal G}$ is called a \emph{presentation} (or \emph{cover}) of $X_{{\mathcal
G}}$. When $G$ is a finite graph and hence compact, $X_{{\mathcal G}}={\mathcal L}_{\infty}(X_{G})$ is called a \emph{sofic shift}. Call $F(I)=\{u:\,u$ is the label of some paths starting at $I\}$ the \emph{follower set} of $I$.

 An irreducible sofic shift is called \emph{almost-finite-type} (AFT)
 if it has a bi-closing presentation \cite{LM}. A (possibly reducible) sofic shift that has a bi-closing presentation is called an \emph{almost Markov}
shift.  A shift space $X$ has \emph{specification with variable gap length} (SVGL) if there exists $N \in \mathbb N$ such that for all $u,v \in {\mathcal B}(X)$, there exists $w \in {\mathcal B}(X)$ with $uwv \in {\mathcal B}(X) $ and $|w|\leq N$.

A word $v \in {\mathcal B}(X)$ is \emph{synchronizing} if whenever $uv,\,vw\in{\mathcal B}(X)$,
we have $uvw \in {\mathcal B}(X)$. An
irreducible shift space $X$ is a \emph{synchronized system} if it has a synchronizing word. Let ${\mathcal G}=(G,{\mathcal L})$ be a labeled graph. A word $ w\in\mathcal B(X_{\mathcal G}) $
is a \emph{magic} 
word if all paths in $  G$ presenting $  w$ terminate at the
same vertex.

A labeled graph ${\mathcal G}=(G,{\mathcal L})$ is \emph{right-resolving} if for each vertex $I$ of $G$ the edges starting at $I$ carry
different labels. A \emph{minimal right-resolving presentation} of a sofic shift $X$ is a right-resolving presentation
 having the fewest vertices among all right-resolving presentations of $X$. It is unique up to isomorphism \cite[Theorem 3.3.18]{LM} and is called the \emph{Fischer cover} of $ X $.

Now we review the concept of the Fischer cover for a not necessarily sofic system developed in \cite{FF}.
Let $ x\in\mathcal B(X) $ and call $ x_{+}=(x_{i})_{i\in\mathbb Z^{+}} $ (resp. $ x_{-}=(x_{i})_{i<0} $) the \emph{right (resp. left) infinite $ X $-ray}. For $ x_{-} $, its follower set is defined as $ \omega_{+}(x_{-})=\{x_{+}\in X^{+}:\,x_{-}x_{+}$ is a point in $X\} $.
Consider the collection of all follower sets $ \omega_{+}(x_{-})$ as the set of vertices of a graph $ X^{+} $. There is an edge from $ I_{1} $ to $ I_{2} $ labeled $ a $ if and only if there is a $ X $-ray $ x_{-} $ such that $ x_{-}a $ is a $ X $-ray and $I_{1}=\omega_{+}(x_{-}) $, $ I_{2}=\omega_{+}(x_{-}a) $. This labeled graph is called the \emph{Krieger graph} for $ X $.
If $ X $ is a synchronized system with synchronizing word $ \alpha $, the irreducible component of the Krieger graph containing the vertex $ \omega_{+}(\alpha)  $ is called the \emph{(right) Fischer cover} of $ X $.

The entropy of a shift space $X$ is defined by $h(X)=\lim_{n
\rightarrow \infty}(1/n)\log|{\mathcal B}_{n}(X)|$. A component $ X_{0} $ of
a shift space $  X$ is called \emph{maximal} if $h(X_{0})=h(X)  $.

\section{open, Semi-open and irreducible Codes}\label{properties}

We start with some necessary lemmas.
\begin{lem}\label{circ}\cite[Lemma 1.4]{BBS}
Assume $ \phi:X\rightarrow Y $ and $ \psi:Y\rightarrow Z$ are surjections such that
$ \psi\circ\phi $ is semi-open. Then, $  \psi$ is semi-open as well. Moreover, if $  \psi$ is irreducible,
then also $  \phi$ is semi-open.
\end{lem}
A
continuous map $\phi:X\rightarrow Y  $  is \emph{almost 1-to-1} if the set of the points $ x\in X $ such that $ \phi^{-1}(\phi(x))=\{x\} $ is dense in $  X$.
\begin{lem}\label{lem:irr}
Let $ \phi:X\rightarrow Y $ and $ \psi:Y\rightarrow Z$ be two factor codes such that
$ \psi\circ\phi $ is irreducible. Then, $  \phi$ and $  \psi$ are irreducible.
\end{lem}
\begin{proof}
If $\phi$ or $ \psi $ is not irreducible, then it is not almost 1-to-1 \cite[Theorem 10.2]{Wh} and  so $ \psi\circ\phi $ cannot be almost 1-to-1 either, a contradiction.
\end{proof}
\begin{lem}\cite[Lemma 1.2]{J1}\label{lem:open}
A code $ \phi:X\rightarrow Y $ between shift spaces is open if and only if for each
$  l\in\mathbb N$, there is $k\in\mathbb N  $ such that whenever $  x\in X$, $y\in Y$ and $ \phi(x)_{[-k,\,k]}=y_{[-k,\,k]} $, we
can pick $ \overline{x}\in X $ with $ \overline{x}_{[-l,\,l]}=x_{[-l,\,l]} $ and $ \phi(\overline{x})=y $.
\end{lem}
So if $  \phi$ is open, then for any $l\in\mathbb N  $ we can find $k\in\mathbb N  $ such that the image of
a central $ 2l+1 $ cylinder in $  X$ consists of central $ (2k+1) $ cylinders in $  Y$.
 We say that $  \phi$ has a \emph{uniform
lifting length} if for each $ l\in\mathbb N $, there exists $  k$ satisfying the above property such
that $ \sup_{l}|k-l|<\infty $.
 Let us formally state a definition for a semi-open code.
\begin{defn}
A code $ \phi:X\rightarrow Y $ between shift spaces is semi-open if  for each
$  l\in\mathbb N$, there is $k\in\mathbb N  $ such that the image of
a central $ 2l+1 $ cylinder in $  X$ contains  a central $ (2k+1) $ cylinder in $  Y$.
\end{defn}
Note that the definition of uniform
lifting length extends naturally to semi-open codes.
\begin{lem}\label{onto} 
Let $  X$ and $  Y$ be shift spaces with $  Y$ irreducible. If $ \phi:X\rightarrow Y $ is a
semi-open code, then it is onto and so a factor code.
\end{lem}
\begin{proof}
Since $ \phi $ is semi-open, $ \phi(X) $ contains a non-empty open set $ U\subseteq Y $. Therefore,  since
$ (Y,\,\sigma) $ is topologically transitive, $ \overline{\cup_{n=-\infty}^{\infty}\sigma^{n}(U)}=Y $ \cite{W}. On the other hand, $ \phi(X) $ is a $  \sigma$-invariant subset of $  Y$. Thus $  \phi$ is onto.
\end{proof}
Our next results give situations where the factor code on covers are semi-open.
 In all of them, irreducibility is an important hypothesis.
 The following example shows that non-semi-openness may occur quite easily in reducible covers.
\begin{example}\label{not semi} 
Let $ X $ be a shift space presented by Figure~\ref{reducible} and $ Y $ the golden shift given by Figure~\ref{even}. 
Define $\Phi : {\mathcal B}_{1}(X) \rightarrow
\{0,1\}$ as,
$$
 \Phi(w)=\left\{
\begin{tabular}{ll}
$1$&$w=1,$ \\
$0$&$\rm{otherwise},$ \\
\end{tabular}
\right .
$$
and let $\phi : X \rightarrow Y$ be the
code induced by $\Phi$. Since $ [2]=2^{\infty} $, $ \phi([2])=0^{\infty} $ and this shows that $\phi$ is not semi-open.
Note that $\phi$ is a finite-to-one code.
\end{example}
\begin{figure}
\begin{center}
\hspace{-25mm}\begin{tikzpicture}[->,>=stealth',shorten >=1pt,auto,node distance=1.5cm,
                    thick,main node/.style={rectangle,draw,font=\sffamily\small\bfseries}]

\node[main node](3) {};

 \path[every node/.style={font=\sffamily\small}]
   
(3)edge [loop above] node {$2$} (3);
\end{tikzpicture}
\end{center}
\vspace{-13mm}
\begin{center}
\begin{tikzpicture}[->,>=stealth',shorten >=1pt,auto,node distance=1.5cm,
                    thick,main node/.style={rectangle,draw,font=\sffamily\small\bfseries}]

  \node[main node](1) {};
\node[main node] (2) [ right of=1] {};
\node[main node](3) {};

 \path[every node/.style={font=\sffamily\small}]
    (1)edge [bend left] node[right] {$\hspace{-4mm}\ ^{\ ^{1}}$} (2)
    edge [loop above] node {$0$} (1)
    (2)  edge [bend left] node[right] {$\hspace{-4mm}\ ^{\ ^{0}}$} (1);
\end{tikzpicture}
\end{center}
\begin{center}
\caption{The cover of $ X $.}\label{reducible} 
\end{center}
\end{figure}
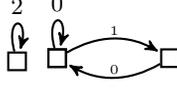
\begin{figure}
\begin{center}
\begin{tikzpicture}[->,>=stealth',shorten >=1pt,auto,node distance=1.5cm,
                    thick,main node/.style={rectangle,draw,font=\sffamily\small\bfseries}]

  \node[main node](1) {};
\node[main node] (2) [ right of=1] {};

 \path[every node/.style={font=\sffamily\small}]
    (1)edge [bend left] node[right] {$\hspace{-4mm}\ ^{\ ^{1}}$} (2)
    edge [loop above] node {$0$} (1)
    (2)  edge [bend left] node[right] {$\hspace{-4mm}\ ^{\ ^{0}}$} (1); 
\end{tikzpicture}
\end{center}
\begin{center}
\caption{The Fischer cover of the golden shift.}\label{even} 
\end{center}
\end{figure}
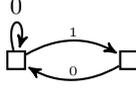

\begin{thm}\label{r-r magic} 
A coded system $X$ is synchronized if and only if there is a cover 
$\mathcal G=(G,\,\mathcal L)$ so that $\mathcal L_{\infty}$ is semi-open. In fact, any irreducible right-resolving cover of a 
synchronized system with a magic word is semi-open.
\end{thm}
\begin{proof}
Let $X$ be synchronized with a magic word $\alpha$ and let $\mathcal G=(G,\,\mathcal L)$ be an irreducible right-resolving cover
for $X$. 
Suppose $ [\pi] =[e_{-l}\cdots e_{l}]\in\mathcal B(X_{G}) $   is a central $ 2l+1 $ cylinder in $X_G$ with $ \mathcal L(e_{i})=a_{i} $, $ -l\leq i\leq l $; thus
 $ \mathcal L_{\infty}([\pi])\subseteq [a_{-l}\cdots a_{l}] $.
Choose
 $\lambda\in\mathcal B(X_{G}) $ such that $ \mathcal L(\lambda)=\alpha $. 
Since $ X_{G} $ is irreducible, there are two paths $ \gamma $ and $ \xi $ in $ G $ such that $ \pi'=\lambda\xi\pi\gamma\lambda\in\mathcal B(X_{G}) $. 
Now set $ w=\mathcal{L}( \pi')=\alpha\mathcal L(\xi)a_{-l}\cdots a_{l}\mathcal L(\gamma)\alpha $ and note that
by the fact that $\alpha$ is magic and $\mathcal G$ is right-resolving,
 if
 $w=\mathcal{L}(\pi'')$, then $\pi''=\pi'$. So  $ [w]\subseteq\mathcal L_{\infty}([\pi'])  $ and we are done.

For the converse let $\mathcal G=(G,\,\mathcal L)$ be a cover for $X$ with $\mathcal L_{\infty}$ semi-open. Since 
$X_G$ is a Polish space, $ \mathcal L_{\infty}(X_G) $ is analytic and so there is an open set $O\subseteq X$ and a first category set $P\subseteq X$ so that $\mathcal L_{\infty}(X_G)=O\triangle P$. Therefore, 
$$\mathcal L_{\infty}(X_G)=\cup_{n\in\mathbb Z}\sigma^{n}(O\triangle P)\supseteq\cup_{n\in\mathbb Z}\sigma^{n}(O)- \cup_{n\in\mathbb Z}\sigma^{n}(P).$$
Since $\mathcal L_{\infty}$ is semi-open, $O$ is non-empty. This means that the first union on right is a non-empty open dense  set.
As a result, $\mathcal L_{\infty}(X_G)$ is residual and $X$ synchronized \cite[Theorem 1.1]{FF}.
\end{proof}

When $\mathcal G$ is the Fischer cover for a synchronized system, we will have a stronger conclusion.
\begin{thm}\label{Fischer} 
If $ X $ is a synchronized system with the Fischer cover 
$ \mathcal G=(G, \mathcal L) $, then $ \mathcal L_{\infty} $ is irreducible. 
\end{thm}
\begin{proof}
Let $ F $ be a  proper closed subset  of $X_G$.
Then,  $F'=X_{G} \setminus F $ is open and there is a doubly transitive point $ x\in F' $. But the degree of $ \mathcal L_{\infty} $ is one \cite[Proposition 9.1.6]{LM} and so $ \mathcal L_{\infty}(x)\not\in  \mathcal L_{\infty}(F)$ which  means that $\mathcal L_{\infty}(F)\neq X$.
\end{proof}
For sofics, when considering finite covers, a rather weaker hypothesis holds the same conclusion as Theorem \ref{r-r magic}.
\begin{thm}\label{finite cover} 
Suppose $ X $ is  sofic and 
$ \mathcal G=(G,\, \mathcal L) $  an irreducible finite cover. Then, $ \mathcal L_{\infty} $ is semi-open.
\end{thm}
\begin{proof}
Suppose $\mathcal L_{\infty}  $ is not semi-open. So there is a
 cylinder $ [u_0] $ in $ X_{G} $ such that $ \mathcal L_{\infty}([u_0]) $ does not have any central cylinder.
 Without loss of generality, assume that $ |u_0|=1 $. 
Then, there must exist another edge $ u'_0\in\mathcal E(G) $ such that $  \mathcal L(u'_0)= \mathcal L(u_0) $; otherwise, $  [\mathcal L(u_0)]\subseteq\mathcal L_{\infty}([u_0]) $ and this contradicts our assumption. 
But to rule out $  [\mathcal L(u_0)]\subseteq\mathcal L_{\infty}([u_0]) $ completely, one must have either $ F(\mathcal L(u_0))\quad\subsetneq F(\mathcal L(u'_0))$ or $ P(\mathcal L(u_0))\quad\subsetneq P(\mathcal L(u_0'))$.
Assume the former and suppose $ f\in F(u_0) $. Then, there must be an
 edge $   f'\in F(u_0')$ such that $  \mathcal L(f')= \mathcal L(f) $ and yet  another path
$   f''\in F(u_0')$ such that $ \mathcal L(f'')\not\in F(I_0) $ where $ I_{0}=t(u_0) $. 

Now choose any path $u_1=u_0w_0u'_0$, $ t(u_1)=t(u'_0)=I_1  $ and observe that since we cannot have $[\mathcal{L}(u_1)]\subseteq \mathcal L_{\infty}([u_0])$, there is another path $u'_1$
labeled as $u_1$  whose last edge is $u''_1$ with $ \mathcal L(u''_1)=\mathcal L(u'_0)=\mathcal L(u_0) $ such that $I_2=t(u''_1)$ has at least one path more than $I_1$. 
Construct $I_n$ inductively. Then, the outer edges of $I_n$ increases by $n$ which is absurd for a finite graph such as $G$.
\end{proof}
The following example shows that 
Theorem~\ref{finite cover}
 does not hold for an infinite irreducible cover.
\begin{example}\label{non semi2} 
Let $ \mathcal G=(G,\,\mathcal L) $ be the cover shown in Figure~\ref{nonsemi-open} and $ e\in\mathcal E(G) $ with $ i(e)=I_{1} $ and $ t(e)=I_{0} $. 
Note that $ X_{\mathcal G} $ is a full-shift on $ \{0,\,1,\,2\} $. 
One way to see this is that for $r\in\mathbb{N}$ and $ m_i\in\mathbb N_{0} $  and any $a_j\in\{0,\,1,\,2\}  $, 
we have a path labeled $ a_1^{m_1}a_2^{m_2}\cdots a_r^{m_r} $.
This implies that
 for any periodic point $ p^{\infty}\in\{0,\,1,\,2\}^{\mathbb Z}  $, 
we have a path $ \pi $ such that $ \mathcal L(\pi)=p $. The conclusion follows since the periodic points are dense in our space.

Now suppose $ \mathcal L_{\infty}([e]) $ contains a central $ 2n+1 $ cylinder $[b_{-n}\cdots b_{n}]$ with $ b_{0}=0 $. By the fact that our system is a full-shift, the point $x=\cdots2^{n+2}b_{-n}\cdots.0\cdots b_{n}\cdots\in[b_{-n}\cdots b_{n}] $; however $ x\not\in\mathcal L_{\infty}([e]) $.
 This is because, we have $ i(b_{-n})=I_{i},\,-n+1\leq i\leq n+1 $. For instance, if $ i=n+1 $, then $ b_{-1}=b_{-2}=\cdots b_{-n}=0 $ and if $ 0\leq i< n+1  $, then at least one $ b_{j} $ equals $ 2 $.
 Now let $A_{j}=\{i\in\mathbb N:\,2^{i}\in P(I_{j})\}$ for $j\in\mathbb Z$ where $ P(I_{j}) $ is the collection of labels of paths
terminating at $ I_{j} $. Then, $ A_{-1}=A_{0}=\emptyset $ and $\max A_{j}=j$ for 
$j\in\mathbb N$ and $\max A_{j}=-j-1$
for $j\in\{-n:\,n\geq2\}$. 
So if a path $ \pi $ is labeled $ 2^{n+2}b_{-n}\cdots.0 $, 
then $ t(\pi)\neq I_{0} $. This means that $\phi$ is not semi-open.
\end{example}

\begin{figure}
\begin{center}
\begin{tikzpicture}[->,>=stealth',shorten >=1pt,auto,node distance=1.5cm,
                    thick,main node/.style={rectangle,draw,font=\sffamily\small\bfseries}]

  \node[](1) {$\cdots$};
\node[main node] (2) [ right of=1] {\SMALL$I_{-3} $};
\node[main node] (3) [ right of=2] {\SMALL$I_{-2} $};
\node[main node] (4) [ right of=3] {\SMALL$I_{-1} $};
\node[main node] (5) [ right of=4] {\SMALL$I_{0} $};
\node[main node] (6) [ right of=5] {\SMALL$I_{1} $};
\node[main node] (7) [ right of=6] {\SMALL$I_{2} $};
\node[main node] (8) [ right of=7] {\SMALL$I_{3} $};
\node[] (9) [ right of=8] {$ \cdots $};

 \path[every node/.style={font=\sffamily\small}]
    (1)edge node [right] { $\hspace{-6mm}\ ^{\ ^{0}}$} (2)
    (2)  edge [bend left] node[right] {$\hspace{-4mm}\ ^{\ ^{1}}$} (1)
    edge [bend right] node[right] {$\hspace{-4mm}\ ^{\ ^{2}}$} (1)
    edge node [right] { $\hspace{-6mm}\ ^{\ ^{0}}$} (3)
    (3)  edge [bend left] node[right] {$\hspace{-4mm}\ ^{\ ^{1}}$} (2)
    edge [bend right] node[right] {$\hspace{-4mm}\ ^{\ ^{2}}$} (2)
    edge node [right] { $\hspace{-6mm}\ ^{\ ^{0}}$} (4)
    (4)  edge [bend left] node[right] {$\hspace{-4mm}\ ^{\ ^{1}}$} (3)
    edge [bend right] node[right] {$\hspace{-4mm}\ ^{\ ^{2}}$} (3)
    edge node [right] { $\hspace{-6mm}\ ^{\ ^{0}}$} (5)
    (5)  edge [bend left] node[right] {$\hspace{-4mm}\ ^{\ ^{1}}$} (4)
    edge [bend left] node[right] {$\hspace{-4mm}\ ^{\ ^{2}}$} (6)
    (6)  edge [bend left] node[right] {$\hspace{-4mm}\ ^{\ ^{2}}$} (7)
    edge node [right] { $\hspace{-6mm}\ ^{\ ^{0}}$} (5)
    (7)  edge [bend left] node[right] {$\hspace{-4mm}\ ^{\ ^{2}}$} (8)
    edge node [right] { $\hspace{-6mm}\ ^{\ ^{0}}$} (6)
    (8) edge [bend left] node[right] {$\hspace{-4mm}\ ^{\ ^{2}}$} (9)
    edge node [right] { $\hspace{-6mm}\ ^{\ ^{0}}$} (7)
(9) edge node [right] { $\hspace{-6mm}\ ^{\ ^{0}}$} (8) ; 
\end{tikzpicture}
\end{center}
\begin{center}
\caption{A cover of the full shift on $\{0,\,1,\,2\}$ which is not semi-open.}\label{nonsemi-open} 
\end{center}
\end{figure}
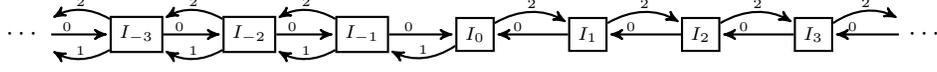

\begin{cor}\label{new} 
Suppose $  X $ is an irreducible sofic  and let $ \phi:X\rightarrow Y $ be a factor code. Then, $ \phi $ is semi-open.
\end{cor}
\begin{proof}
Let $ \mathcal G=(G, \mathcal L) $ be the Fischer cover of $ X $. Then, $ \phi\circ\mathcal L _{\infty} $ is an irreducible
finite cover for $ Y $ and by Theorem \ref{finite cover}, it is semi-open. Now the proof is a consequence of  Lemma~\ref{circ}.
\end{proof}

\begin{thm}\label{transitive sys} 
Let $T:X\rightarrow X$ and $S:Y\rightarrow Y$ be transitive homeomorphisms and $\phi:X\rightarrow Y$ a factor map. 
Suppose there is an open set $U_0$ such that for any open set $V\subseteq U_0$, int$ (\phi(V))\neq\emptyset $. Then, $\phi$
is semi-open.
\end{thm}
\begin{proof}
Let $U$ be an arbitrary open set. 
Since $ T $ is transitive, there is an integer $ n $ such that $ U_{0}\cap T^{-n}(U) $ is a non-empty open set. By hypothesis, there is open set $ V\subseteq \phi(U_{0}\cap T^{-n}(U))  $ and for that
$$ S^{n}(V)\subseteq S^{n}\phi(U_{0}\cap T^{-n}(U))=\phi T^{n}(U_{0}\cap T^{-n}(U))=\phi (T^{n}(U_{0})\cap U)\subseteq\phi(U). $$
and since $ S $ is homeomorphism, $ S^{n}(V) $ is open.
\end{proof}
An immediate consequence of this theorem is that if $ X $ and $ Y $ are subshifts with $ T $ and $ S $ shift maps, then image of any cylinder under $ \phi $ has empty interior if and only if there is just one cylinder in $ X $ with empty interior under $ \phi $. Example \ref{non semi2} is a good example to see this fact.

\section{Open   vs semi-open codes}\label{vs}

 The next two theorems say that 
\cite[Proposition 3.3]{J1} and \cite[Proposition 3.4]{J1} that were stated for open codes  are actually valid for 
 semi-open codes as well. The main ingredient in the proof is to have an open set lying in the image of any open set which is supplied by any semi-open map.

\begin{thm}\label{SFT finite to one}
Let $  X$ be a shift of finite type, $  Y$ an irreducible sofic shift and
$\phi: X\rightarrow Y$ a finite-to-one semi-open code. Then, $  X$ is non-wandering and all components are
maximal.
\end{thm}
Example \ref{not semi} shows that semi-openness is required in the hypothesis of the above theorem.
\begin{thm}\label{right closing}
Let $ \phi:X\rightarrow Y $ be a right-closing semi-open code from a shift space
$  X$ to an irreducible shift of finite type $  Y$. Then, $  X$ is a non-wandering shift of finite
type.
\end{thm}
By Theorem~\ref{r-r magic}, unlike open codes \cite[Corollary 4.3]{J1}, a right-closing 
semi-open extension of an irreducible strictly almost Markov (resp. non-AFT sofic) shift 
is not necessarily strictly almost Markov (resp. non-AFT sofic).

We have some similarities between open and semi-open codes. For instance, a coded system $X$ is SFT if and only if there is a cover 
$\mathcal G=(G,\,\mathcal L)$ so that $\mathcal L_{\infty}$ is open which can be deduced from \cite[Proposition 3.1.6]{LM} and \cite[Proposition 2.3]{J1}. A similar result for semi-open codes is Theorem~\ref{r-r magic}. Also, SFT is invariant under open codes as synchronized systems under semi-open codes:
\begin{thm}\label{s to s}
Suppose $X$ is a synchronized system and $\phi:X\rightarrow Y$ a semi-open code. Then, $Y$ is  synchronized.
\end{thm}
\begin{proof}
The Fischer cover of $X$ will be a cover for $Y$ whose cover map is semi-open; this fact and the conclusion is supplied by Theorem~\ref{r-r magic}.
\end{proof}

Note that  any coded system $Y$ can be a factor  of some synchronized systems \cite[Proposition 4.1]{BH2}. 
In particular, if $Y$ is not synchronized, then by the above theorem, the
associated code is not semi-open. 

%
%
%
%
\begin{example}
Unlike codes between irreducible sofic shifts which are all semi-opens (Corollary \ref{new}), there can be a non-semi-open code between two synchronized systems.
For an example, let $Y'$ be a non-synchronized system on $\mathcal{A'}=\{0,\,1,\ldots ,\,k-1\}$ generated by a prefix code $C'_{Y'}=\{ua',\,w_1,\ldots\}$,
 $a'\in\mathcal{A'}$ as in \cite[Proposition 4.1]{BH2}. This means if $v\in C'_{Y'}$, then no other element of $C'_{Y'}$ starts with $v$
 and we know that any coded system has a prefix code \cite{BH2}. 
Let $\phi':X'\to Y'$ be 
the one-block code between the synchronized system $X'$ generated by $C'_{X'}=\{uz,\,w_1,\ldots\}$, 
$z\not\in \mathcal{A'}\cup\{k\}=\{0,\,1,\ldots ,k\}$  and $\phi'$ induced by the block map $ \Phi' $ defined as
\begin{equation}
 \Phi'(x)_i=\left\{
\begin{tabular}{ll}
$x_i$&$x_i\in\mathcal A',$ \\
$a'$&$x_i=z.$ \\
\end{tabular}
\right .
\end{equation} 
By Theorem \ref{s to s}, $\phi'$ is not semi-open.

Now let $X$ and $Y$ be the synchronized systems 
generated by $C_X=C'_{X'}\cup \{k\}$ and $C_Y=C'_{Y'}\cup \{k\}$ respectively. 
Here, $k$ is a synchronizing word for both systems.
Also, define $\phi:X\to Y$ as $\phi'$ when restricted to $X'$ and in other places by $\Phi(k)=k$.  
Then, since $\phi'$ is not semi-open, $\phi$ is not semi-open either.
\end{example}
Jung in \cite{J1} shows that when
$  \phi$ is a finite-to-one code from a shift of finite type $  X$ to an
irreducible sofic shift $  Y$, then $  \phi$ is open if and only if it is constant-to-one. For semi-open codes we have
\begin{thm}\label{semi a.e.} 
Let $  \phi$ be a finite-to-one code from a shift of finite type $  X$ to an
irreducible sofic shift $  Y$. 
 If $  \phi$ is semi-open, then it will be constant-to-one a.e..
\end{thm}
\begin{proof}
By Theorem \ref{SFT finite to one}, $  X$ is non-wandering and all components are maximal. Since each component is SFT and $ \phi $ is finite-to-one, the restriction of $  \phi$ to each component is constant-to-one a.e. \cite[Theorem 9.1.11]{LM}. So $ \phi $ is constant-to-one a.e. as well.
\end{proof}
\begin{rem}
Example \ref{not semi} shows that the converse of Theorem \ref{semi a.e.} does not hold necessarily. 
However, when
$  X$ is an irreducible shift of finite type and $ \phi: X\rightarrow Y $ a finite-to-one factor code, then, $ \phi $ is semi-open and constant-to-one a.e. (Corollary \ref{new} and \cite[Theorem 9.1.11]{LM}). 
\end{rem}


Let $\phi: X\rightarrow Y$ be an open code between shift spaces. If $  \phi$ is
bi-closing and $  Y$ irreducible, then $  \phi$ is constant-to-one \cite[Corollary 2.8]{J1}.
The following example shows that this result  is not necessarily true  for semi-open codes.
\begin{example}\label{nonopen} 
Let $  Y$ be the even shift with its Fischer cover $\mathcal G=(G,\,\mathcal L) $ in Figure \ref{even}. Note that 
$  \mathcal L_{\infty}$ is bi-resolving and so $  Y$ is an AFT.
 By Theorem \ref{r-r magic}, $  \mathcal L_{\infty}$ is semi-open. 
But $|\mathcal L^{-1}_{\infty}(y)|=1  $ for all $ y\neq 0^\infty $ and equals 2 for $ y=0^{\infty} $.
So $  \mathcal L_{\infty}$ is not constant-to-one. 
\end{example}
Let $  X$, $  Y$ and $  Z$ be shift spaces and $\phi_1:X\rightarrow Z  $, $\phi_2:Y\rightarrow Z  $ the codes. Then,
the \emph{fiber product} of $(\phi_{1},\,\phi_{2})  $ is the triple $ (\Sigma,\,\psi_{1},\,\psi_{2}) $ where
$$\Sigma=\{(x,y)\in X\times Y:\,\phi_{1}(x)=\phi_{2}(y)\}$$
and  $\psi_1:\Sigma\rightarrow X  $ is defined by  $\psi_{1}(x,y)=x $; similarly for $  \psi_{2}$. Let $ \phi_1 $ be onto. Then, if $\psi_1$ is open, then so is $\phi_2$ \cite[Lemma 2.4]{J1}. We do not know this result for semi-openness. However, we have the following. 
\begin{thm}\label{fiber} 
Let $ (\Sigma;\,\phi_1,\,\phi_2) $ be the fiber product of $\phi_1:X\rightarrow Z  $ and $\phi_2:Y\rightarrow Z  $.
If $ \psi_1 $ and $\phi_1$ are semi-open and $ \psi_2 $ is onto, then $ \psi_2 $ and $\phi_2$ are semi-open.
\end{thm}
\begin{proof}
If $ \psi_1 $ and $\phi_1$ are semi-open, then Lemma \ref{circ} 
implies that $\phi_2$ is semi-open. To see that $\psi_2$ is semi-open  let $  C$ be the central $ 2l+1 $ cylinder in $  \Sigma$
and let $  m$ be the coding length of $\phi_2 $.
 Choose $ p\geq l $ so that $\psi_{1}(C)  $ contains  a central $ 2p+1 $ cylinder and let $ x\in[x_{-p}\cdots x_{p}]\subseteq\psi_{1}(C)  $.
 Also, let $ q\geq l $ so that $\phi_{1}([x_{-p}\cdots x_{p}])  $ contains  a central $ 2q+1 $ cylinder and let 
 $ z\in[z_{-q}\cdots z_{q}]\subseteq\phi_{1}([x_{-p}\cdots x_{p}]) $. Then, $ z=\phi_{1}(\overline{x}) $ 
where $ \overline{x}\in[x_{-p}\cdots x_{p}] $. 
So there is $ \overline{y}\in Y $ such that $ (\overline{x},\,\overline{y})\in C $. 

Now we claim that $  [\overline{y}_{-q-m}\cdots\overline{y}_{q+m}]\subseteq \psi_{2}(C)$. 
For $ y'_{[-q-m,\,q+m]}=\overline{y}_{[-q-m,\,q+m]} $,
 we have $ \phi_{2}(y')_{[-q,\,q]}=\phi_{2}(\overline{y})_{[-q,\,q]} $ 
which means that $\phi_{2}(y')\in[z_{-q}\cdots z_{q}]  $. So there is $ x'\in[x_{-p}\cdots x_{p}] $ such that $ \phi_{1}(x')=\phi_{2}(y') $ and this in turn means that $ (x',\,y')\in \Sigma $. 
But $ y'_{[-l,\,l]}=\overline{y}_{[-l,\,l]} $ and so in fact
$(x',\,y')\in C$.
  Thus  $\psi_{2}(C) $ contains $  [\overline{y}_{-q-m}\cdots\overline{y}_{q+m}]$ and the claim is established.
\end{proof}
\subsection{Lifting the semi-open code between synchronized systems  to a code between their Fischer covers} 
We usually read most of the properties of a synchronized system from its Fischer cover, so lifting the codes between systems to codes between their respective Fischer   covers could be helpful.
\begin{lem}\label{thm:lift}
Let $ X_i $, $ i=1,\,2 $ be a synchronized system with the Fischer cover $ \mathcal G_i=(G_i,\,\mathcal L_i) $. Then, a code $f:X_1\rightarrow X_2$ lifts to the Fischer covers. That is, there is a code
 $F:X_{G_1}\rightarrow X_{G_2}$ such that ${\mathcal {L}_2}_{\infty}\circ F=f\circ {\mathcal{ L}_1}_{\infty}$.
\end{lem}
\begin{proof}
If $t\in X_1$ is doubly transitive, then $f(t)$ is doubly transitive in $X_2$ and so 
by the fact that the degree of $ {\mathcal L_i}_\infty $ is one \cite[Proposition 9.1.6]{LM},
 $ F={\mathcal {L}_2}_{\infty}^{-1}\circ f\circ {\mathcal {L}_1}_{\infty} $ makes sense on doubly transitive points. 
Pick $ x\in X_{G_1} $  and choose a sequence $ \{x_n\}_{n\in\mathbb N} $ of doubly transitive points converging to $x$. 
We will show that $\lim_{n\rightarrow\infty}F(x_n) $ exists.
First note that $ F(x_n) $ is a doubly transitive point for all $ n\in\mathbb N $.
Let $u=\mathcal{L}_2(\pi)$ be a synchronizing word for $X_2$ and $w$ any word  such that $uwu\in\mathcal{B}(X_2)$. 
If  $ \pi $ terminates at a vertex $ I $, then since $\mathcal L_{2_{\infty}}  $ is right-resolving, there is a unique path in $G_2$ with initial vertex $ I $  representing $w$.
Also, any  word occurs in a doubly transitive point infinitely often to the left and to the right; therefore, arbitrary central paths of $F(x_n)$ and $F(x_m)$ are the same for sufficiently large $m$ and $n$ and since we are on paths,  $\lim_{n\rightarrow\infty}F(x_n) $ exists. 
Set $ F(x):=\lim_{n\rightarrow\infty}F(x_n) $.

Now we claim that $ F $ is continuous. Let $ \{z_n \}_{n} $ be a sequence in $ X_{G_{1}} $ such that $ z_{n}\rightarrow z $. It is sufficient to show that $  F(z_n)\rightarrow F(z) $. 

For all $ n $, suppose $ \{w^{(n)}_m \}_{m} $ is a sequence of doubly transitive points in $ X_{G_{1}} $ such that $ w^{(n)}_m\rightarrow z_{n} $. Then, we have
$$\lim_{(m,\,n)\rightarrow\infty}w^{(n)}_m=z.$$
So by definition, $ F(z)=\lim_{(m,\,n)\rightarrow\infty}F(w^{(n)}_m) $. On the other hand, for all $ n $, $ F(z_n)=\lim_{m\rightarrow\infty}F(w^{(n)}_m) $ and this proves the claim and we are done.
\end{proof}
\begin{thm}\label{liftt} 
By the hypothesis of the above lemma,
\begin{enumerate}
\item
$f$ is irreducible if and only if $F$ is.
\item
If $F$ is semi-open, then $f$ is semi-open. When $ G_1$ and $G_2$ are locally-finite, the converse is also true.
\end{enumerate}
\end{thm}
\begin{proof}\ \ 
 (1)    is a consequence of Lemma~\ref{lem:irr} and Theorem~\ref{Fischer}.

For the proof of (2),  let $ F $ be semi-open and $U$ an open subset of $X_1$ and set   $$g:= {\mathcal{L}_2}_\infty\circ F\, (=f\circ {\mathcal{L}_1}_\infty).$$
 Then, there is  an open set $V\subseteq g({\mathcal{L}_1}^{-1}_\infty(U))$. But
 $ g({\mathcal{L}_1}^{-1}_\infty(U))\subseteq f(U)$ and so $V\subseteq f(U)$ and consequently
 $f$ is semi-open.

For the converse suppose $f$ is semi-open and let $U$ be an open subset of $X_{G_1}$. Since by Theorem \ref{r-r magic},
$g$ is semi-open, we may assume that there is $V$, an open subset of $X_2$, such that
 $g(U)=V$; otherwise, replace $U$ with $U\cap g^{-1}(V)$, $V\subseteq g(U)$.
Also, by the fact that $X_{G_i}$ is locally compact, we may assume that $\overline U$ is a compact subset of $X_{G_1}$ and so $X_{G_2}\setminus F(\overline{U})$ an open subset of $X_{G_2}$. 
Furthermore, we choose $U$ so small so that $X_{G_2}\setminus F(\overline{U})\ne \emptyset$.
By these assumptions, we will show that $W=F(U)$ is open in $X_{G_2}$. 
Suppose the contrary. Thus there is  $w\in W$ which is a boundary point of $W$ and then indeed $w\in \partial F(\overline{U})$.
Choose a sequence of doubly transitive points $\{x_n\}$ in $X_{G_2}\setminus F(\overline{U})$
 approaching $w$. 
Since $X_{G_2}\setminus F(\overline{U})$ is open, such a sequence exists and
$$\lim_{n\rightarrow \infty}{\mathcal{L}_2}_\infty(x_n)={\mathcal{L}_2}_\infty(w)\in V.$$
 But  ${\mathcal{L}_2}_\infty$ is one-to-one on the set of doubly transitive points which means that ${\mathcal{L}_2}_\infty(x_n)\in X_2\setminus V$ and this in turn implies that our open set $V$ contains a boundary point which is absurd.
\end{proof}

Now let $ X $ be a non-SFT sofic shift with the Fischer cover $ \mathcal G=(G,\,\mathcal L) $ and let $ f=\mathcal L_{\infty} $ and $ F=Id:X_{G}\rightarrow X_{G} $. 
Then, $ F $ is open while $ f $ cannot be open \cite[Proposition 2.3]{J1}.
 So Theorem \ref{liftt} does not hold for open maps.

\subsection{On the bi-continuing codes and bi-closing codes}
A code $ \phi:X\rightarrow Y  $ is called \emph{right-continuing almost everywhere (a.e.)} if whenever $x$ is left transitive in $X$ and $\phi(x)$ is left
asymptotic to a point $y\in Y$, then there exists $\overline{x}\in X$ such that $\overline{x}$ is left asymptotic
to $x$ and $\phi(\overline{x})=y$. Similarly we have the notions called left-continuing a.e. and
bi-continuing a.e.

If $\phi: X\rightarrow Y  $ is open with a uniform lifting length, 
then it is bi-continuing (everywhere) with a bi-retract \cite[Lemma 2.1]{J2}. 
The following shows that a semi-open code with a uniform lifting length is not necessarily bi-continuing a.e. with a bi-retract.

\begin{example}
This example was constructed in \cite{J2} to show that a continuing code may not have a retract; we use it for our prementioned purpose. 

 Let $  X$ be a shift space on the alphabet
$ \{1,\,\overline{1},\,2,\,3\} $ defined by forbidding $\{\overline{1}2^{n}3:\,n\geq0\}  $, and let $  Y	$ be the full $  3$-shift
$ \{1,\,2,\,3\}^{\mathbb Z} $. Define $\phi=\Phi_{\infty}: X\rightarrow Y  $ by letting $  \Phi(\overline{1})=1$ and $  \Phi(a)=a$ for all $ a\neq{\overline 1 }$. The map 
$  \phi$ is bicontinuing and it has (left  continuing) retract \cite{J2}.
 Now for $n\in\mathbb{N}$, consider a left transitive point $ x=\cdots\overline{1}2^{n}.22^{\infty}$ and pick
 $ y=\cdots12^{n}.23^{\infty}\in Y$
 so that $ \phi(x) $ is left asymptotic to $ y $. One can readily
see that $  \phi$ does not have right  continuing a.e. retract.

We show that $ \phi $ is semi-open with a uniform lifting length. Let $ [a_{-n}\cdots a_{n}] $ be a central $ 2n+1 $ cylinder in $X$ and $\Phi(a_{i})=b_{i}$, $-n\leq i\leq n$. Thus
if $ a_{i}\neq\overline{1} $, then $ \phi( [a_{-n}\cdots a_{n}] )= [b_{-n}\cdots b_{n}]  $;
otherwise, $ [1b_{-n}\cdots b_{n}1]\subseteq 
[b_{-n}\cdots b_{n}1]\subseteq\phi( [a_{-n}\cdots a_{n}] )$. So $ \phi $ is semi-open with a uniform lifting length as required.
\end{example}
If $X$ is an irreducible sofic shift, then by \cite[Corollary 4.4.9]{LM},
\begin{equation}\label{subsystem}
h(Y)<h(X), \  Y\mbox{ is a proper subsystem of } X.
\end{equation}
This condition is sufficient to have double transitivity a  totally invariant property  for finite-to-one factor codes. That is,
\begin{lem}\label{doubly}
Suppose $  X$ is compact and $\phi: X\rightarrow Y  $ a finite-to-one factor code. If  either $  X$  satisfies \eqref{subsystem} or $ \phi $ is irreducible, then $  x\in X$ is doubly transitive if and 
only if $ \phi(x)$ is.
\end{lem}
\begin{proof}
When $  X$  satisfies \eqref{subsystem}, it is a consequence of \cite[Theorem 3.4]{A1}. So suppose $ \phi $ is irreducible. If $\phi(x)  $ is doubly transitive, but $  x$ is not, then the proof of \cite[Lemma 9.1.13]{LM} implies that there will be a proper subshift $ Z $ of $ X $ with $ \phi(Z)=Y $; and this contradicts the irreducibility of $ \phi $.
\end{proof}
\begin{thm}\label{bi-retract} 
Suppose $X$ and $\phi$ satisfy the hypothesis of Lemma
 \ref{doubly} and
 $ Y $ a SFT.
 If $ \phi $ is bi-continuing a.e. with a bi-retract, then it will be open with a uniform lifting length.
\end{thm}
\begin{proof}
Suppose that $  \phi$ is bi-continuing a.e. with a bi-retract $ n\in\mathbb N $. 
We may assume that our SFT system
 $  Y$ is
 an edge shift and $  \phi$ is
 a $  1$-block code. 
Let $ l\geq0 $ and $ u\in\mathcal B_{2l+1}(X) $.
 Choose $ x\in[u] $ to be a left transitive point and pick a doubly transitive point
 $ y\in Y $ with $ y_{[-l-n,\,l+n]}=\phi(x)_{[-l-n,\,l+n]} $. 
The point $ \hat{y} $ given by
$$\hat{y}_{i}=\left\{
\begin{tabular}{ll}
$\phi(x)_{i}$&$i\leq0,$ \\
$y_{i}$&$i>0.$ \\
\end{tabular}
\right .$$
is a doubly transitive point in $Y  $. Since $  n$ is a right continuing a.e. retract and $ \phi(x)_{(-\infty,\,l+n]}=\hat{y}_{(-\infty,\,l+n]} $,
 there is
$ \overline{x}\in X $ such that $\overline{x}_{(-\infty,\,l]}=x_{(-\infty,\,l]}  $ and $ \phi(\overline{x})=\hat{y}$. 
By Lemma \ref{doubly}, $ \overline{x} $ is a doubly transitive point and we have
 $ \phi(\overline{x})_{[-l-n,\,\infty)}=\hat{y}_{[-l-n,\,\infty)}=y_{[-l-n,\,\infty)} $. 
By the fact that  $n$ is also a left continuing a.e. retract, there is  $ \overline{\overline{x}}\in X $ such that
$ \overline{\overline{x}}_{[-l,\,\infty)}=\overline{x}_{[-l,\,\infty)} $ and $  \phi(\overline{\overline{x}})=y$.
 Note that $\overline{\overline{x}}_{[-l,\,l]}=\overline{x}_{[-l,\,l]}=x_{[-l,\,l]} $ which means that $ \overline{\overline{x}}$ is  in $[u] $. 

Now choose an arbitrary $y'\in [\phi(x)_{[-l-n,\,l+n]}]$ and pick  a sequence of doubly
 transitive points $ \{y^{m}\}_{m} $ in $[\phi(x)_{[-l-n,\,l+n]}]  $ such that $ y^{m}\rightarrow y' $ and let $\overline{\overline{x}}^m$ be a point in $[u]$ with $\phi(\overline{\overline{x}}^m)=y^m$. Thus  for all $ m $, 
$ y^{m}\in\phi([u]) $. But $ \phi([u]) $ is closed  and so 
$ y'\in\phi([u]) $ which implies that $[\phi(x)_{[-l-n,\,l+n]}]\subseteq\phi([u])  $. This shows that $\phi$ is semi-open.
 
It remains to prove that $\phi$ is open. 
Let again $u=u_{-l}\cdots u_l$ 
and  let $ x'\in[u] $ and choose a doubly transitive point
 $ x $   such that $ x_{[-l-n,\,l+n]}=x'_{[-l-n,\,l+n]} $. 
Since $ \phi $ is a $  1$-block code, $ \phi(x)_{[-l-n,\,l+n]}=\phi(x')_{[-l-n,\,l+n]} $. Also since $ Y $ is an edge shift, $[\phi(x')_{[-l-n,\,l+n]}] =[\phi(x)_{[-l-n,\,l+n]}]$. By above, $    [\phi(x)_{[-l-n,\,l+n]}]\subseteq\phi([u]) $. 
But $x'$ was arbitrary and so 
$  \phi$ is open with uniform lifting length.
 \end{proof}
By Theorem \ref{bi-retract} and \cite[Lemma 2.1]{J2}, we have the following:
\begin{cor}\label{cor}
Let $X$ and $\phi$ satisfy the hypothesis of Lemma \ref{doubly} and
 $ Y $ a SFT. If $ \phi $ is bi-continuing a.e. with a bi-retract, then it will be bi-continuing (everywhere) with a bi-retract.
\end{cor}
There is another situation where a bi-continuing a.e. code implies bi-continuing. In fact, Ballier  interested in sofics in \cite{B}, 
states and proves the next theorem for when $ X $ is an irreducible sofic and $ k=0 $. 
His proof exploits only the irreducibility of $ X $  which is provided here.
Hence we have
\begin{thm}\label{Ballier} 
Suppose $  X$ is an irreducible shift space and $  Y$  an irreducible SFT.
A right-continuing a.e. factor map $\phi:X\rightarrow Y  $ with retract $ k $
 is right-continuing
(everywhere) with retract $ k $. 
\end{thm}
The next theorem gives
a sufficient condition for a factor code being bi-continuing a.e..
\begin{thm}\label{syn bi-continuing a.e.} 
Let $ X $ be synchronized and $ \phi:X\rightarrow Y $ a semi-open code. Then, $ \phi $ is bi-continuing a.e..
\end{thm}
\begin{proof}
Let $x=\cdots x_{-1}x_{0}x_{1}\cdots\in X$ be a left transitive point and $y\in Y$ such that $ \phi(x)_{(-\infty,\,0]}=y_{(-\infty,\,0]} $. Since $x$ is left transitive, any word in $X$ happens infinitely many often on the left of $x$. So without loss of generality, assume that $x_0$ is the synchronizing word. 

By semi-openness, there is a cylinder $ _{l}[u] $ contained in $ \phi(_{0}[x_0]) $ where $ _{l}[u] $ denotes the set 
$ \{y\in Y:\,y_{[l,\,l+|u|-1]}=u\} $. Since $x$ is left transitive, there exist infinitely many $k>0$ such that $ \sigma^{-k}(\phi(x))\in {_{l}}[u] $,
or equivalently, $ \phi(x)\in \sigma^{k}({_{l}}[u])={_{l-k}}[u] $. Choose a  sufficiently large $ k $ so that $ l-k+|u|-1<0 $. 

Now since $ \phi(x)_{(-\infty,\,0]}=y_{(-\infty,\,0]} $, $ y\in {  _{l-k}}[u] $ or in fact $y\in _{l-k}[v]$
 where $$v=y_{l-k}\ldots y_{-1}y_0=u_{l-k}\ldots u_{l-k+|u|-1}y_{l-k+|u|}\ldots y_{-1}y_0.$$
We have ${_{l-k}}[v]\subseteq {_{l-k}}[u]\subseteq\phi(_{-k}[x_0]) $.
Set $W=\phi^{-1}( _{l-k} [v])\cap {_0[x_0]}$ and note that
 $W$ is an open set containing $x$ and $y\in \phi(W)$.
 Therefore,  there must be $ z\in W$ such that $ \phi(z)=y $ and
define 
$$
 \overline{x}_i=\left\{
\begin{tabular}{ll}
$x_i$&$i\leq 0,$ \\
$z_i$&$i\geq 0.$ \\
\end{tabular}
\right .
$$
Since $ x_0 $ is a synchronizing word, $ \overline{x}\in X $ and so we have a $\overline{x}$ which 
is left-asymptotic to $ x $ and $ \phi(\overline{x})=y $. This means $ \phi $ is right-continuing a.e.; and similarly, it is left-continuing a.e..
\end{proof}

\bibliographystyle{amsplain}

\begin{thebibliography}{90}


\bibitem{A1}
D. Ahmadi Dastjerdi and S. Jangjoo, On synchronized non-sofic subshifts, arXiv:1303.6570.
\bibitem{A2}
E. Akin, E. Glasner, W. Huang, S. Shao and X. Ye, Sufficient conditions under which a transitive system is chaotic,  \textit{Ergo. Th. \& Dynam. Sys}. \textbf{30}, No. 5, 1277-1310 (2010).
\bibitem{Au}
J. Auslander, \textit{Minimal flows and their extensions}, North-Holland Math. Stud.
153, North-Holland, 1988.
\bibitem{B}
 Ballier, A.: Limit sets of stable cellular automata, Ergo. Th. \& Dynam. Sys, DOI: 10.1017/etds.2013.72, (2013).
\bibitem{BBS}
A. Bella, A. Błaszczyk, A. Szymański: On absolute retracts of $\omega ^\ast $. \textit{Fund. Math}. \textbf{145} (1994), 1-13.
\bibitem{BH2}
F. Blanchard and G. Hansel, Syst\`emes cod\'es, \textit{Theoretical Computer Science}. \textbf{44} 14-49, 1986.
\bibitem{Boyle}
M. Boyle and S. Tuncel, Infinite-to-one codes and Markov measures, \textit{Trans. Amer. Math. Soc}.
\textbf{285} (1984), 657-684.
\bibitem{FF}
D. Fiebig and U. Fiebig, Covers for coded systems,
\textit{Contemporary Mathematics}, \textbf{135}, 1992, 139-179.
\bibitem{J2}
U. Jung, On the existence of open and bi-continuing codes, \textit{Trans. Amer. Math.
Soc}. \textbf{363} (2011), 1399–1417.
\bibitem{J1}
U. Jung, Open maps between shift spaces, \textit{Ergo. Th. \& Dynam. Sys}. \textbf{29} (2009),
1257–1272.
\bibitem{LM}
 D. Lind and B. Marcus, \textit{An introduction to symbolic dynamics and coding}, Cambridge Univ. Press, 1995.
\bibitem{SM}
T. K. Subrahmonian Moothathu, Studies in topological dynamics with emphasis on cellular automata, PhD thesis, Department of Mathematics and Statistics, School of MCIS, University of Hyderabad, (2006).
\bibitem{W}
P. Walters, An introduction to ergodic theory, Springer-Verlag, 1982.
\bibitem{Wh}
G. T. Whyburn, Analytic topology, \textbf{28}, AMS Coll. Publications, Providence, RI, 1942.





\end{thebibliography}

\end{document}